\documentclass[11pt,table]{article}

\usepackage{mathpazo}
\usepackage{amsfonts}
\usepackage{amsmath}
\usepackage{graphicx}
\usepackage{latexsym}
\usepackage{booktabs}
\usepackage{multicol}

\usepackage[margin=1.05in]{geometry}
  
\usepackage[T1]{fontenc}
\usepackage{times}
\usepackage{array}
\usepackage{enumerate}
\usepackage{amsmath}
\usepackage{amssymb}
\usepackage{amsthm}

\newtheorem{lemma}{Lemma}
\newtheorem{theorem}{Theorem}

\usepackage{xcolor}
\usepackage{makeidx}
\usepackage[colorlinks=true,linkcolor=blue,anchorcolor=blue,citecolor=red,urlcolor=magenta]{hyperref}

\newcommand{\tr}{\operatorname{Tr}}

\newcommand{\defeq}{\stackrel{\smash{\textnormal{\tiny def}}}{=}}

\makeatletter
\newcommand\footnoteref[1]{\protected@xdef\@thefnmark{\ref{#1}}\@footnotemark}
\makeatother

\def\C{\mathbb{C}}

\def\Q{\mathbb{Q}}
\def\R{\mathbb{R}}
\def\P{\mathbb{P}}
\def\dM{\mathbb{M}}

\def\0{\mathbf{0}}

\begin{document}

\title{Real Schur norms and Hadamard matrices}  

\author{
	John Holbrook\footnote{Department of Mathematics \& Statistics, University of Guelph, Guelph, ON, Canada N1G 2W1}
\and	
	Nathaniel Johnston\footnote{Department of Mathematics \& Computer Science, Mount Allison University, Sackville, NB, Canada E4L 1E4}\textsuperscript{$\ \ *$}
\and
Jean-Pierre Schoch\footnote{89 Hayes Ave., Guelph, ON, Canada N1E 5V7}
}

\date{June 5, 2022}

\maketitle

\begin{abstract}
    We present a preliminary study of Schur norms $\|M\|_{\textup{S}}=\max\{ \|M\circ C\|: \|C\|=1\}$, where $M$ is a matrix whose entries are $\pm1$, and $\circ$ denotes the entrywise (i.e., Schur or Hadamard) product of the matrices. We show that, if such a matrix $M$ is $n\times n$, then its Schur norm is bounded by $\sqrt{n}$, and equality holds if and only if it is a Hadamard matrix. We develop a numerically efficient method of computing Schur norms, and as an application of our results we present several almost Hadamard matrices that are better than were previously known.
\end{abstract}
\let\thefootnote\relax\footnotetext{E-mails: jholbroo@uoguelph.ca, njohnston@mta.ca, jp\_schoch@yahoo.co.uk}
  
\section{Introduction}\label{sec:intro}

A student new to matrix analysis might be surprised to note that the most commonly used norm for matrices (often called the ``operator norm''), i.e.,
$$
    \|M\| \defeq \max\{\|M\mathbf{u}\|:\|\mathbf{u}\|=1\},
$$
behaves somewhat perversely: it is not generally diminished when the matrix entries are diminished (in modulus). The simplest examples suffice:
\[
    \left\|\begin{bmatrix}1&-1\\1&1\end{bmatrix}\right\|=\sqrt{2}\approx 1.4142, \quad \text{while} \quad \left\|\begin{bmatrix}1&-1\\0&1\end{bmatrix}\right\|=\frac{1+\sqrt{5}}{2}\approx 1.6180.
\]
 
One might then ask how extreme this effect can be, i.e., what is the value of
\begin{equation} \label{E1}
    \max_{A,B}\{\|B\|/\|A\|:|b_{i,j}|\leq|a_{i,j}| \ \text{for all} \ i,j\}?
\end{equation}
It turns out that the answer to this question depends on the matrix size and is easy when working with complex matrices, but very hard for real matrices: the first author noticed some years ago that a complete answer in the real case would also settle the existence question for Hadamard matrices (see the upcoming Theorem~\ref{thm:c_c and r_n}). While this connection is unlikely to help resolve that existence question, it does suggest that the study of question~\eqref{E1} will be challenging and worthwhile. The present note initiates that study.

Our setting is the Hilbert space $\C^n$ or $\R^n$, i.e., the space of column vectors of dimension $n$, furnished with the usual Euclidean inner product and norm. We denote by $\mathbb{M}_n(\C)$ and $\mathbb{M}_n(\R)$ the algebras of $n\times n$ complex and real matrices, respectively. These act on the relevant Hilbert spaces by matrix-vector multiplication, and the operator norms $\|M\|$ are simply the norms of the corresponding linear mappings. A real matrix $M$ may be regarded as acting on either $\C^n$ or $\R^n$ and it is a standard fact that the operator norm is independent of this choice.

\section{Schur Norms}\label{sec:schur_norms}

The question~\eqref{E1} can be conveniently framed in the context of Schur norms, which we now introduce. Given $n\times n$ matrices $M$ and $C$, we denote by $M\circ C$ their entrywise (i.e., ``Schur'' or ``Hadamard'') product, and the \emph{Schur norm} $\|M\|_{\textup{S}}$ of $M$ is defined by
 \begin{equation} \label{E2}
	\|M\|_{\textup{S}} = \max_{C \in \mathbb{M}_n(\C)}\{\|M\circ C\|:\|C\|=1\}.
\end{equation}
Let $\P_n(\R)$ (and $\P_n(\C)$) denote the set of matrices $M \in \mathbb{M}_n(\R)$ (or in $\mathbb{M}_n(\C)$) such that all $|m_{i,j}|\leq 1$ for all $1 \leq i,j \leq n$. Then, in terms of Schur norms, question~\eqref{E1} is asking for the values of
\begin{equation} \label{E3}
	c_n \defeq \max\{\|M\|_{\textup{S}} : M\in\P_n(\C)\},
\end{equation}
and
\begin{equation} \label{E4}
	r_n \defeq \max\{\|M\|_{\textup{S}} : M\in\P_n(\R)\}.
\end{equation}
We will show that, while $c_n$ is easy to evaluate, knowing the values of $r_n$ for all $n$ would resolve the famous conjecture about existence of Hadamard matrices (i.e., matrices with mutually orthogonal columns and entries $\pm 1$) when $n$ is a multiple of $4$.

While Schur norms are more difficult to compute than the operator norm, there are a number of useful techniques that may be applied. The simplest case occurs when $M$ is positive semidefinite; Schur showed in 1911 \cite{Sch11} that in this case we have $\|M\|_{\textup{S}} = \max_j \{m_{j,j}\}$. Another simple-to-evaluate case was established in \cite{Mat93}, where it was shown that if $M$ is a circulant matrix with top row $c_0$, $c_1$, $\ldots$, $c_{n-1}$, then
\begin{equation} \label{E6}
	\|M\|_{\textup{S}} = \frac{1}{n}\sum_{k=0}^{n-1}|p(\omega^k)|,
\end{equation}
where $\omega=e^{2\pi i/n}$ and $p(z)=c_0+c_1z+\dots+c_{n-1}z^{n-1}$.

Schur also showed that if $M$ is any matrix (rectangular or square) then we have
\begin{equation} \label{E5}
	\|M\|_{\textup{S}} \leq \min\{\|R\|_{\textup{r}}\|C\|_{\textup{c}}: RC=M\},
\end{equation}
where $\|R\|_{\textup{r}}$ denotes the ``row norm'' of $R$, i.e., the maximum norm of its rows, and $\|C\|_{\textup{c}}$ similarly denotes the largest of the norms of the columns of $C$. Many years later, Grothendieck (and independently Haagerup) showed that we actually have equality in~\eqref{E5}; see the discussion in \cite{DD07}. An elegant treatment of these results may be found in \cite[Section~1.4 and Chapter~3]{Bha97}.

While the formula~\eqref{E5} can be useful for obtaining upper bounds on Schur norms, it does not provide a practical means for computation, since it is not clear how to perform the minimization. In Section~\ref{sec:compute_schur}, we will use a result of Paulsen, Power, and Smith to develop a method that allows numerical approximation to any accuracy (and sometimes exact evaluation) of Schur norms via semidefinite programming. Our method allows for the computation of Schur norms in polynomial time, contrary to the commonly-repeated statement that computation of this norm is hard \cite{Hla99,DD07}.

\section{Schur Witnesses}\label{sec:witnesses}

For a given matrix $M$, we call a contraction $C$ (i.e., a matrix with $\|C\|\leq 1$) such that $\|M\circ C\|=\|M\|_{\textup{S}}$ a \emph{witness} for the Schur norm of $M$. More precise forms
for the witness are available, as the following theorem explains.

\begin{theorem}\label{thm:witness forms}
    \begin{itemize}
        \item[(1)] If $M\in\dM_n(\R)$ then there is a real witness for $\|M\|_{\textup{S}}$.
        \item[(2)] If $M\in\dM_n(\R)$ then there is a (real) orthogonal witness for $\|M\|_{\textup{S}}$.
        \item[(3)] If $M\in\dM_n(\C)$ then there is a unitary witness for $\|M\|_{\textup{S}}$.
    \end{itemize}
\end{theorem}

We note that condition~(1) of this theorem was proved in \cite[Corollary~3.3]{Mat93b}, along with many other useful results related to Schur norms. We give another, more direct, proof of this result.

\begin{proof}[Proof of Theorem~\ref{thm:witness forms}]
    (1) Let $M\in\dM_n(\R)$. We establish the result by showing that for any contraction $C \in \dM_n(\C)$, there exists a contraction $R$ in $\dM_n(\R)$ such that $\|M\circ C\|\leq \|M\circ R\|$.
    
    To this end, let $\mathbf{u}$ be a norming vector for $M\circ C$, i.e., $\|\mathbf{u}\|=1$ and $\|(M\circ C)\mathbf{u}\|=\|M\circ C\|$. Let $\mathbf{w}$ be the unit vector $(M\circ C)\mathbf{u}/\|M\circ C\|$, so that $\|M\circ C\|=\langle (M\circ C)\mathbf{u}, \mathbf{w}\rangle$. Let $D$ be a diagonal unitary such that $D\mathbf{u}=|\mathbf{u}|$ where $|\mathbf{u}|_k=|u_k|$, and 
    let $E$ be a diagonal unitary such that $E\mathbf{w}=|\mathbf{w}|$. Then $\|M\circ C\|=\langle(M\circ C)D^*|\mathbf{u}|,E^*|\mathbf{w}|\rangle$. Since $D$ and $E$ are diagonal,
    $E(M\circ C)D^*=M\circ(ECD^*)$, so that
    $$
        \|M\circ C\| = \langle (M\circ (ECD^*))|\mathbf{u}|,|\mathbf{w}|\rangle.
    $$
    The only (possibly) non-real element in the above equation is $Y=ECD^*$, so we also have
    \[
        \|M\circ C\|=\langle (M\circ R)|\mathbf{u}|,|\mathbf{w}|\rangle,
    \]
    where $R$ is the real part of $Y$. Since $\|R\|\leq \|Y\|=\|C\|\leq 1$, we conclude that and that $\|M\circ C\|\leq\|M\circ R\| \||\mathbf{u}|\| \||\mathbf{w}|\| = \|M\circ R\|$, completing the proof of~(1).
    
    (2) The function $f(C) := \|M\circ C\|$ is convex with respect to the variable $C$ in the unit ball of contractions in $\dM_n(\R)$. Thus $f$ attains its maximum on some extreme point of that unit ball, i.e., at some orthogonal matrix.
    
    (3) Argue as in the proof of (2), but noting that the extreme points of the unit ball of $\dM_n(\C)$ are the unitary matrices (in fact, in the complex case any contraction is the average of just two unitaries).
\end{proof}

It is perhaps worth noting that, in the proof of (1) above, we actually get $\|M\circ C\| = \|M\circ R\|$, not just $\|M\circ C\| \leq \|M\circ R\|$. To see that $\|M\circ R\|\leq \|M\circ C\|$ as well, notice that
\[
    M\circ R=(M\circ Y + M\circ \overline{Y})/2,
\]
$\|M\circ \overline{Y}\|=\|\overline{M\circ Y}\|=\|M\circ Y\|$, and $\|M\circ Y\|=\|E(M\circ C)D^*\|=\|M\circ C\|$.

The next lemma is no doubt ``well-known'', although we do not know of a reference. It can be proved by an elementary argument based on the Cauchy--Schwarz inequality and the corresponding conditions for equality.

\begin{lemma}\label{lem:well--known}
	Let $M$ be an $n\times m$ matrix with columns $\mathbf{c}_1$, $\mathbf{c}_2$, $\ldots$, $\mathbf{c}_m$ such that $\|\mathbf{c}_j\|\leq 1$ for all $j$. Then $\|M\| \leq \sqrt{m}$, with equality if and only if $\|\mathbf{c}_1\|=1$ and there exist scalars $z_1$, $z_2$, $\ldots$, $z_m$ with $|z_j| = 1$ and $\mathbf{c}_j=z_j \mathbf{c}_1$ for all $j$.
\end{lemma}

The following theorem establishes the main connection between Schur norms and Hadamard matrices. In particular, it gives the only upper bound on $r_n$ that we currently know how to compute, with the exception of exact values that we can compute when $n$ is small with extensive computer help (see Section~\ref{sec:compute_rn}).

\begin{theorem}\label{thm:c_c and r_n}
    Let $n \geq 1$ be an integer and let $c_n$ and $r_n$ be as in Equations~\eqref{E3} and~\eqref{E4}, respectively.
    \begin{itemize}
        \item[(1)] $c_n = \sqrt{n}$.
        \item[(2)] $r_n \leq \sqrt{n}$.
        \item[(3)] $r_n=\sqrt{n}$ if and only if there exists an $n\times n$ Hadamard matrix.
    \end{itemize}
\end{theorem}

\begin{proof}
    Let $\|M\|_{\textup{F}}$ denote the Frobenius (or Hilbert--Schmidt) norm of the $n\times n$ matrix $M$, i.e.,
	$$
	    \|M\|_{\textup{F}}=\sqrt{\sum_{i,j=1}^n |m_{i,j}|^2}.
	$$
	It is well-known that $\|M\|\leq \|M\|_{\textup{F}}\leq \sqrt{n}\|M\|$. Thus, if $M\in \P_n(\C)$ (or $\P_n(\R)$), we have
	\begin{align}\label{eq:frob_norm_reduc}
	    \|M\circ B\|\leq \|M\circ B\|_{\textup{F}} \leq \|B\|_{\textup{F}} \leq \sqrt{n}\|B\|,
	\end{align}
	which proves $r_n \leq \sqrt{n}$ (i.e., (2)) and $c_n \leq \sqrt{n}$.
	
	To prove~(1), we need to find a matrix in $\P_n(\C)$ with Schur norm at least $\sqrt{n}$. To this end, let $F_n\in \P_n(\C)$ denote the matrix whose $(i,j)$-th entry is $\omega^{ij}$, where $\omega=e^{2\pi i/n}$; this matrix is often called the ($n$-dimensional) Fourier matrix. Notice that the columns of $F_n$ are mutually orthogonal and each have norm $\sqrt{n}$. It follows that $B = \overline{F_n}/\sqrt{n}$ is	unitary and thus has $\|B\| = 1$. Furthermore, $F_n\circ B$ is the matrix with every entry equal to $1/\sqrt{n}$, so $\|F_n\|_{\textup{S}} \geq \|F_n\circ B\| = \sqrt{n}$ (Lemma~\ref{lem:well--known} could be invoked here, for example). This completes the proof of~(1).
	
	To prove~(3), let $H_n$ be an $n\times n$ Hadamard matrix. Then $H_n\in\P_n(\R)$ and we see that $\|H_n\|_{\textup{S}}=\sqrt{n}$ by the same argument that we have used for $F_n$ above. All that remains is to show that if $r_n=\sqrt{n}$ then there exists an $n\times n$ Hadamard matrix. Note first that the convexity of the Schur norm ensures that among the $M\in\P_n(\R)$ such that $\|M\|_{\textup{S}} = r_n$, we may choose $M$ to be an extreme point in the convex set $\P_n(\R)$. That is, if $\Q_n(\R)$ denotes the set of $n \times n$ matrices with entries equal to $\pm 1$, then
    \begin{align}\label{eq:rn_reduced}
    	r_n = \max\{\|M\|_{\textup{S}} : M\in\Q_n(\R)\}.
    \end{align}
    
    We then know from Theorem~\ref{thm:witness forms} that there exists an orthogonal $B \in \mathbb{M}_n(\mathbb{R})$ such that $\sqrt{n} = r_n = \|M\|_{\textup{S}} = \|M \circ B\|$. Since the columns of $A := M \circ B$ all have norm~$1$, Lemma~\ref{lem:well--known} tells us that the entries of $B$ do not vary in modulus along rows. Since the transpose $B^T$ is a witness for $M^T$, the entries of $B$ are also invariant in modulus along columns. We thus conclude that $B$ is a real orthogonal matrix whose entries are $\pm 1/\sqrt{n}$ so that $\sqrt{n}B$ is Hadamard. This completes the proof of (3). 
\end{proof}

It is perhaps worth noting that Bhatia, Choi, and Davis \cite[Proposition~3.1]{BCD89} showed that multiples of the Frobenius norm are the only unitarily invariant matrix norms that have the norm-reducing property of Inequality~\eqref{eq:frob_norm_reduc}.

We could also consider Schur norms for rectangular $n\times m$ matrices $M$, and define, in the obvious way, the quantities $c_{n,m}$ and $r_{n,m}$. Since we have $\|M\|\leq\|M\|_{\textup{F}}\leq\min\{\sqrt{n},\sqrt{m}\}\|M\|$, we could argue, as in the proof above, that
$$
    r_{n,m} \leq c_{n,m} = \min\{\sqrt{n},\sqrt{m}\}.
$$

\section{Schur Spectra and Inequivalent Matrices}\label{sec:spectra}

In the proof of Theorem~\ref{thm:c_c and r_n}, we introduced the set
\[
    \Q_n(\R) \defeq \{M\in\dM_n(\R):m_{i,j} \in \{-1,1\} \ \text{for all} \ 1 \leq i,j \leq n\}.
\]
Its importance comes from Equation~\eqref{eq:rn_reduced}, which says that we can compute $r_n$ just by finding which member of this (finite!) set has largest Schur norm. We use the fanciful term \emph{$n$-th Schur spectrum} for the set of \emph{all} Schur norms of members of $\Q_n(\R)$:
$$
    \sigma_n \defeq \{\|M\|_{\textup{S}} : M\in\Q_n(\R)\}.
$$
This set $\sigma_n$ is a finite collection of real numbers between $1$ and $\sqrt{n}$ (inclusive). We now present some other basic properties of $\sigma_n$. In the following theorem, condition~(1) merely restates our observation about $\Q_n(\R)$ from the proof of Theorem~\ref{thm:c_c and r_n}.

\begin{theorem}\label{thm:spectra}
    \begin{itemize}
        \item[(1)] $r_n = \max \sigma_n$.
        \item[(2)] $\sigma_n\subseteq\sigma_{n+1}$.
        \item[(3)] $\sigma_n\sigma_m\subseteq\sigma_{nm}$, where $\sigma_n\sigma_m := \{xy : x \in \sigma_n, y \in \sigma_m\}$.
    \end{itemize}
\end{theorem}

\begin{proof}
    (1) The convex function $f(M):=\|M\|_{\textup{S}}$ on $\P_n(\R)$ attains its maximum at some extreme point, i.e., some element of $\Q_n(\R)$.
    
    (2) Consider $x\in\sigma_n$ and let $M\in\Q_n(\R)$ be such that $x=\|M\|_{\textup{S}}$. By the equality form of Inequality~\eqref{E5} (i.e., the Grothendieck--Haagerup result discussed in Section~\ref{sec:schur_norms}), there exist $n\times n$ matrices $R$ and $C$ such that $M=RC$ and $x=\|R\|_{\textup{r}}\|C\|_{\textup{c}}$. Let $R_{+}$ be the $(n+1)\times n$ matrix obtained by repeating the top row of $R$, and let $C_{+}$ be the $n\times(n+1)$ matrix obtained by repeating the first column of $C$. Then $R_{+}C_{+}$ is $M$ bordered on the top and left by certain elements of $M$, so $R_{+}C_{+}\in \Q_{n+1}(\R)$. Inequality~\eqref{E5} shows that $\|R_{+}C_{+}\|_{\textup{S}}\leq\|R_{+}\|_{\textup{r}}\|C_{+}\|_{\textup{c}} = \|R\|_{\textup{r}}\|C\|_{\textup{c}}=x$. On the other hand, $R_{+}C_{+}$ has $M$ as a 
    submatrix, so  $\|R_{+}C_{+}\|_{\textup{S}}\geq\|M\|_{\textup{S}}=x$, which can be seen by augmenting any witness for $\|M\|_{\textup{S}}$ with a row and column of zeros. It follows that $\|R_{+}C_{+}\|_{\textup{S}} = x$, so $x \in \sigma_{n+1}$.
    
    (3) Let $x\in\sigma_n$ and $y\in\sigma_m$ and suppose $A\in\Q_n(\R)$ and $B\in\Q_m(\R)$ are such that $x=\|A\|_{\textup{S}}$ and $y=\|B\|_{\textup{S}}$. The Kronecker product $A\otimes B$ is an element of $\Q_{nm}(\R)$ and it is known that $\|A\otimes B\|_{\textup{S}}=\|A\|_{\textup{S}}\|B\|_{\textup{S}}$ (see \cite[Proposition~1]{Hla99}, for example). It follows that $xy\in\sigma_{nm}$.
\end{proof}

For small values of $n$, we are able to compute $\sigma_n$ explicitly, with help from computer software. In particular, we are able to give an exact analytic description of $\sigma_n$ when $n \leq 5$, and we are able to provide numerical approximations of all members of $\sigma_n$ when $n \leq 7$. The next subsection explains how these computations were performed.

\subsection{(In)Equivalent Matrices}\label{sec:equiv_mat}

Because $\Q_n(\R)$ and the operator norm are invariant under the operations of changing the sign of a row or column, permuting rows or columns, and transposition, it is straightforward to see that the Schur norm is also invariant under these operations. We say that elements of $\Q_n(\R)$ related by such operations are \emph{equivalent}, and we note that this really does define an equivalence relation, and thus a partition of $\Q_n(\R)$ into disjoint equivalence classes. To compute $r_n$ or $\sigma_n$, it suffices to compute the Schur norm of just a single member of each of these equivalence classes.

Since we are free to multiply each row and column of a member of $\Q_n(\R)$ by $-1$ without changing its equivalence class, we know that each equivalence class contains a matrix whose top row and left column consists entirely of ones. We typically (but not always) work with representatives of equivalences classes that have this form.

With the help of computer software, the equivalence classes for $n \leq 5$ can be enumerated straightforwardly by brute force computation involving all members of $\Q_n(\R)$. We can extend this enumeration to $n = 6$ and $n = 7$ by working recursively as follows: to find the equivalence classes in $\Q_{n+1}(\R)$, add a single extra row and column of $\pm 1$ entries to a single representative for each equivalence class in $\Q_{n}(\R)$ in all possible ways. We can further reduce the amount of computation required by noting that matrices in the same equivalence class have the same singular values as each other (though the converse does not hold, as evidenced by the $5 \times 5$ matrices
\[
    \begin{bmatrix}
         1 &    1 &    1 &    1 &    1 \\
         1 &    1 &   -1 &   -1 &   -1 \\
         1 &   -1 &   -1 &   -1 &    1 \\
         1 &   -1 &   -1 &   -1 &   -1 \\
         1 &   -1 &    1 &    1 &   -1
    \end{bmatrix} \quad \text{and} \quad \begin{bmatrix}
         1 &    1 &    1 &    1 &    1 \\
         1 &    1 &   -1 &   -1 &   -1 \\
         1 &   -1 &   -1 &   -1 &    1 \\
         1 &   -1 &   -1 &   -1 &    1 \\
         1 &   -1 &    1 &    1 &   -1
    \end{bmatrix},
\]
which are inequivalent but have the same singular values). This results in far fewer matrices and equivalences to check than just directly looping over all of $\Q_{n+1}(\R)$. In particular, for $n = 1, 2, \ldots, 7$, the MATLAB code that we provide at \cite{SuppCode} shows that there are $1$, $2$, $3$, $10$, $30$, $242$, $4386$ equivalence classes inside $\Q_n(\R)$ \cite{oeisA353052}, which gives an upper bound on the number of members of $\sigma_n$.

\subsection{Small Schur Spectra}\label{sec:small_schur_spectra}

For $n=2$, there are only two equivalence classes in $\Q_2(\R)$, which contain the matrices
\begin{align}\label{eq:M_mat_r2}
    M_1=\begin{bmatrix}1&1\\1&1\end{bmatrix} \quad \text{and} \quad M_2=\begin{bmatrix}1&1\\1&-1\end{bmatrix}.
\end{align}
It is clear that $M_1$ (indeed, a matrix of any size whose entries all equal $1$) has Schur norm $1$, and $M_2$ is Hadamard and thus has Schur norm $\sqrt{2}$. It follows that $\sigma_2=\{1,\sqrt{2}\}$, so $r_2 = \sqrt{2}$.

When $n=3$, the three equivalence classes in $\Q_3(\R)$ can be represented by the matrices
\begin{align}\label{eq:M_mat_r3}
    M_1= \begin{bmatrix}1&1&1\\1&1&1\\1&1&1\end{bmatrix}, \quad M_2=\begin{bmatrix}1&1&-1\\-1&1&1\\1&-1&1\end{bmatrix}, \quad \text{and} \quad M_3=\begin{bmatrix}1&1&1\\1&-1&-1\\1&-1&-1\end{bmatrix}.
\end{align}
Since $M_2$ is a circulant, the formula~\eqref{E6} tells us that $\|M_2\|_{\textup{S}} = 5/3$. We claim that $\|M_3\|_{\textup{S}}=\sqrt{2}$. To see this, we note that $\|M_3\|_{\textup{S}} \geq \sqrt{2}$ since $M_3$ has the $2\times 2$ Hadamard matrix as a submatrix, while the factorization
\begin{align}\label{eq:M3_norm}
 M_3=\begin{bmatrix}1&0&0\\0&1&0\\0&1&0\end{bmatrix} \begin{bmatrix}1&1&1\\1&-1&-1\\0&0&0\end{bmatrix}
\end{align}
shows that $\|M_3\|_{\textup{S}}\leq\sqrt{2}$ via Inequality~\eqref{E5}. It follows that $\sigma_3=\{1,\sqrt{2},5/3\}$, so $r_3=5/3$.

For larger values of $n$, computing $\sigma_n$ becomes significantly more complicated (though we already know that $r_4 = 2$ since there is a $4 \times 4$ Hadamard matrix). For example, when $n=4$ there are $10$ equivalence classes in $\Q_4(\R)$, represented by the matrices
\begin{align*}
    M_1 & = \begin{bmatrix}
        1 & 1 & 1 & 1 \\
        1 & 1 & 1 & 1 \\
        1 & 1 & 1 & 1 \\
        1 & 1 & 1 & 1
    \end{bmatrix}, & M_2 & = \begin{bmatrix}1 & 1 & -1 & -1 \\
    -1 & 1 & 1 & -1 \\
    -1 & -1 & 1 & 1 \\
     1 & -1 & -1 & 1\end{bmatrix}, & M_3 & = \begin{bmatrix}1 & 1 & 1 & 1\\
    1 & 1 & -1 & -1\\
    1 & -1 & 1 & -1\\
    1 & -1 & -1 & 1\end{bmatrix},\\
    M_4 & = \begin{bmatrix}1 & 1 & 1 & 1\\
    1 & -1 & 1 & 1\\
    1 & -1 & -1 & -1\\
    1 & -1 & -1 & -1\end{bmatrix}, & M_5 & = \begin{bmatrix}1 & 1 & 1 & 1\\
    1 & -1 & -1 & 1\\
    1 & -1 & -1 & -1\\
    1 & -1 & -1 & -1\end{bmatrix}, & M_6 & = \begin{bmatrix}1 & 1 & 1 & 1\\
    1 & -1 & -1 & -1\\
    1 & -1 & -1 & -1\\
    1 & 1 & 1 & 1\end{bmatrix},\\
    M_7 & = \begin{bmatrix}1 & 1 & 1 & 1\\
    1 & -1 & -1 & -1\\
    1 & -1 & -1 & -1\\
    1 & -1 & -1 & -1\end{bmatrix}, & M_8 & = \begin{bmatrix}1 & 1 & 1 & 1\\
    1 & 1 & -1 & 1\\
    1 & -1 & 1 & 1\\
    1 & -1 & -1 & -1\end{bmatrix}, & M_9 & = \begin{bmatrix}1 & 1 & 1 & 1\\
    1 & -1 & -1 & 1\\
    1 & -1 & 1 & 1\\
    1 & -1 & -1 & -1\end{bmatrix},\\
    M_{10} & = \begin{bmatrix}1 & 1 & 1 & 1\\
    1 & 1 & -1 & -1\\
    1 & -1 & 1 & 1\\
    1 & -1 & -1 & -1\end{bmatrix}.
\end{align*}

It is straightforward to see that $\|M_1\|_{\textup{S}} = 1$, $M_2$ is circulant so Equation~\eqref{E6} tells us that $\|M_{2}\|_{\textup{S}} = \sqrt{2}$, and $M_3$ is Hadamard so $\|M_3\|_{\textup{S}} = 2$. Similarly, $M_4$ and $M_5$ are the same (up to equivalences) as the $3 \times 3$ matrix $M_2$ from Equation~\eqref{eq:M_mat_r3}, but with a repeated row and column, so $\|M_{4}\|_{\textup{S}} = \|M_{5}\|_{\textup{S}} = 5/3$, and $M_6$ and $M_7$ are just the $2 \times 2$ Hadamard matrix with repeated rows and columns, so $\|M_{6}\|_{\textup{S}} = \|M_{7}\|_{\textup{S}} = \sqrt{2}$. It's perhaps worth noting that these are the smallest examples to illustrate that members of different equivalence classes in $\Q_n(\R)$ can have the same Schur norm as each other.

However, we have not yet seen any methods that can effectively compute the Schur norms of $M_8$, $M_9$, or $M_{10}$. We solve this problem in Section~\ref{sec:compute_schur}, where we develop an algorithm for efficiently computing the Schur norm of any matrix. In particular, that method shows that
\begin{align*}
    \|M_8\|_{\textup{S}} = (2 + 3\sqrt{6})/5, \quad \|M_9\|_{\textup{S}} = \sqrt{2 + \sqrt{2}}, \quad \text{and} \quad \|M_{10}\|_{\textup{S}} = \sqrt{3},
\end{align*}
so we conclude that $\sigma_4 = \big\{1, \sqrt{2}, 5/3, \sqrt{3}, \sqrt{2 + \sqrt{2}}, (2 + 3\sqrt{6})/5, 2\big\}$.

Since we have computed all equivalence classes in $\Q_n(\R)$ for $n \leq 7$, we could in principal compute $\sigma_5$, $\sigma_6$, and $\sigma_7$ as well. However, this becomes somewhat tricky as $n$ increases, since some of the members of these Schur spectra are difficult to describe analytically. For example, $\sigma_5$ contains $16$ members:
\[
    \sigma_5 = \sigma_4 \cup \big\{(1+4\sqrt{5})/5, (3+8\sqrt{2})/7, 11/5, 1.9093, 1.9621, 2.0130, 2.0276, 2.0591, 2.1343\big\},
\]
but some of these members are just numerical approximations that do not admit nice closed-form expressions. For $n = 6$, we find that the Schur spectrum $\sigma_6$ consists of the following $87$ numbers, to $8$ decimal places of accuracy:
\begin{center}\begin{multicols}{5}
\noindent 1.00000000\\
1.41421356\\
1.66666667\\
1.73205081\\
1.84775907\\
1.86969385\\
1.90934903\\
1.94365063\\
1.96211651\\
1.98885438\\
2.00000000\\
2.01298411\\
2.02756977\\
2.04285840\\
2.04481550\\
2.05198924\\
2.05872925\\
2.05907464\\
2.06472827\\
2.07214975\\
2.08166600\\
2.08479890\\
2.09039387\\
2.09716754\\
2.10439768\\
2.10832063\\
2.11335055\\
2.11745179\\
2.12731473\\
2.13133319\\
2.13264561\\
2.13425873\\
2.13435585\\
2.14421623\\
2.14802944\\
2.15012159\\
2.15083953\\
2.15300969\\
2.15470054\\
2.15493373\\
2.15500794\\
2.15802279\\
2.16028453\\
2.16238205\\
2.16334817\\
2.16890954\\
2.17206182\\
2.17399409\\
2.17546577\\
2.17827029\\
2.17841264\\
2.17849279\\
2.18257226\\
2.18426204\\
2.18624046\\
2.18961912\\
2.19315903\\
2.20000000\\
2.20126537\\
2.20224912\\
2.20355424\\
2.20374983\\
2.20896701\\
2.21177503\\
2.22152260\\
2.22832822\\
2.23579713\\
2.23606798\\
2.23619478\\
2.24319201\\
2.24710210\\
2.24771457\\
2.24919207\\
2.25037290\\
2.25756011\\
2.26081512\\
2.26575640\\
2.27157284\\
2.28263992\\
2.30096841\\
2.30894113\\
2.31264921\\
2.31509300\\
2.33333333\\
2.35566250\\
2.35702260\\
2.38742589.
\end{multicols}\end{center}
For $n = 7$, the Schur spectrum $\sigma_7$ contains approximately $1560$ different numbers, many of which are extremely close to each other. We do not list them here.

\section{Lower Bounds on $r_n$}\label{sec:lower_bounds}

We showed in Theorem~\ref{thm:c_c and r_n} that $r_n \leq \sqrt{n}$, with equality if and only if an $n \times n$ Hadamard matrix exists. We now prove some lower bounds that show that $r_n$ cannot be very far below $\sqrt{n}$.

\begin{lemma}\label{lem:rn_nondec}
    The sequence $\{r_n\}$ is non-decreasing.
\end{lemma}

\begin{proof}
    If $A \in \mathbb{P}_n(\mathbb{R})$ is such that $r_n = \|A\|_{S}$ then we can let $\widetilde{A} \in \mathbb{P}_{n+1}(\mathbb{R})$ be the matrix obtained by padding $A$ with a row and column of zeroes, giving $r_{n+1} \geq \|\widetilde{A}\|_{\textup{S}} = \|A\|_{\textup{S}} = r_n$.
\end{proof}

In fact, the above lemma is also a corollary of condition~(2) of Theorem~\ref{thm:spectra}. When we combine this monotonicity of $\{r_n\}$ with the fact that $r_n = \sqrt{n}$ whenever an $n \times n$ Hadamard matrix exists, we get some simple lower bounds on $r_n$. For example, since there is a $4 \times 4$ Hadamard matrix, we know that $r_4 = \sqrt{4} = 2$, so $r_7 \geq r_6 \geq r_5 \geq r_4 = 2$.

To get a slightly more general lower bound on $r_n$, recall that Sylvester's construction \cite{Syl67} shows that there exist Hadamard matrices of size $m \times m$, and hence $r_m = \sqrt{m}$, whenever $m$ is a power of $2$. Given any $n \in \mathbb{N}$, there exists such an $m$ satisfying $n/2 \leq m \leq n$, so it follows that
\begin{align}\label{eq:sylvester_bound}
    r_{n} \geq r_m = \sqrt{m} \geq \sqrt{n/2}.
\end{align}
A slightly more careful analysis along these lines leads to the following better lower bound on $r_n$:

\begin{theorem}\label{thm:fn_lowerbound}
    If $n \in \mathbb{N}$ and $p$ is a prime number satisfying $p \leq n/2 - 1$, then $r_n \geq \sqrt{2(p+1)}$. In particular,
    \[
        \lim_{n\rightarrow\infty} \frac{r_n}{\sqrt{n}} = 1.
    \]
\end{theorem}

Before we prove this theorem, it is perhaps worth noting that its hypotheses cannot be satisfied unless $n \geq 6$, since smaller values of $n$ give $n/2 - 1 < 2$ and thus no valid prime number $p$ exists.

\begin{proof}[Proof of Theorem~\ref{thm:fn_lowerbound}.]
    The Scarpis construction of Hadamard matrices \cite{Sca98} tells us that if $p$ is prime then there exists a $2(p+1) \times 2(p+1)$ Hadamard matrix, so $r_{2(p+1)} = \sqrt{2(p+1)}$ when $p$ is prime. Since $p \leq n/2 - 1$ is equivalent to $n \geq 2(p+1)$, monotonicity of $\{r_n\}$ (i.e., Lemma~\ref{lem:rn_nondec}), now tells us that $r_n \geq r_{2(p+1)} = \sqrt{2(p+1)}$, establishing the lower bound claimed by the theorem.

    We now prove the limit equality claimed by the theorem. By the prime number theorem, for every $\varepsilon > 0$ there exists $N \geq 1$ such that, for all $n \geq N$, the largest prime $p$ with $p \leq n/2 - 1$ satisfies $p \geq (1-\varepsilon)(n/2 - 1)$. It follows that
    \[
        r_n \geq \sqrt{2(p+1)} \geq \sqrt{2\big((1-\varepsilon)(n/2 - 1)+1\big)} = \sqrt{(1-\varepsilon)n + 2\varepsilon},
    \]
    so dividing both sides by $\sqrt{n}$ shows that
    \[
        \frac{r_n}{\sqrt{n}} \geq \sqrt{1-\varepsilon + \frac{2\varepsilon}{n}} > \sqrt{1-\varepsilon} \quad \text{whenever} \quad n \geq N.
    \]
    Since $\varepsilon > 0$ was arbitrary, it follows that
    \[
        \lim_{n\rightarrow\infty} \frac{r_n}{\sqrt{n}} = 1,
    \]
    as claimed.
\end{proof}

In the above proof, we used the Scarpis construction of Hadamard matrices that is based on prime numbers. This construction was generalized by Paley \cite{Pal33} to prime powers, and thus the lower bound of Theorem~\ref{thm:fn_lowerbound} could be tightened up to make use of prime powers instead. However, just using primes and the Scarpis construction is enough to show that $\lim_{n\rightarrow\infty}r_n/\sqrt{n} = 1$, which was our goal here.

\section{Computation of Schur Norms}\label{sec:compute_schur}

We now introduce a method of efficiently computing the Schur norm $\|A\|_{\textup{S}}$ of a matrix $A \in \mathbb{M}_n(\mathbb{C})$. In particular, the upcoming Theorem~\ref{thm:schur_norm_sdp} presents a semidefinite program whose optimal value is $\|A\|_{\textup{S}}$. We do not introduce the details of how semidefinite programs are solved numerically, or how semidefinite programming duality works. Rather, we simply note that they can be solved numerically in polynomial time \cite{Lov06}, and we direct the reader to any of a number of introductions to the topic like \cite[Section 1.2.3]{Wat18} and \cite[Section~3.C]{JohALA} for more details.

\begin{theorem}\label{thm:schur_norm_sdp}
    Let $A \in \mathbb{M}_n(\mathbb{C})$. Then $\|A\|_{\textup{S}}$ is the optimal value of each of the following semidefinite programs in the variables $X,Y=Y^*,Z=Z^* \in \mathbb{M}_n(\mathbb{C})$, $\mathbf{v},\mathbf{w} \in \mathbb{R}^n$, and $c \in \mathbb{R}$, which are dual to each other:
\begin{align*}
		\begin{matrix}
		\begin{tabular}{r l c r l}
		\multicolumn{2}{c}{\underline{\textup{Primal problem}}} & \quad \quad & \multicolumn{2}{c}{\underline{\textup{Dual problem}}} \\
		\textup{minimize:} & $c$ & \quad \quad & \textup{maximize:} & $\mathrm{Re}\big(\tr(AX^*)\big)$ \\
		\textup{subject to:} & $\begin{bmatrix}
            Y & A \\ A^* & Z
        \end{bmatrix} \succeq O$ & \quad \quad \quad & \textup{subject to:} & $\begin{bmatrix}
            \mathrm{diag}(\mathbf{v}) & X \\ X^* & \mathrm{diag}(\mathbf{w})
        \end{bmatrix} \succeq O$ \\
		\ & $y_{j,j} = z_{j,j} = c \ \textup{for all} \ 1 \leq j \leq n$ & \quad \quad & \ & $\sum_{j=1}^n v_{j} + \sum_{j=1}^n w_{j} \leq 2$.	
		\end{tabular}
		\end{matrix}
\end{align*}
\end{theorem}

\begin{proof}
    It was shown in \cite{PPS89} that $\|A\|_{\textup{S}} \leq 1$ if and only if there exist Hermitian $Y,Z \in \mathbb{M}_n(\mathbb{C})$ with $y_{j,j} = z_{j,j} = 1$ for all $1 \leq j \leq n$ such that
    \[
        \begin{bmatrix}
                Y & A \\ A^* & Z
        \end{bmatrix} \succeq O.
    \]
    It follows immediately from positive homogeneity of $\|A\|_{\textup{S}}$ that $\|A\|_{\textup{S}} \leq c$ if and only if there exist Hermitian $Y,Z \in \mathbb{M}_n(\mathbb{C})$ satisfying the same positive semidefinite condition, but with $y_{j,j} = z_{j,j} = c$ for all $1 \leq j \leq n$. Minimizing over all such $c$, $Y$, and $Z$ is exactly what the primal problem described by the theorem does, and its optimal value is thus $\|A\|_{\textup{S}}$.
    
    The fact that the dual semidefinite program has the indicated form follows from a routine calculation. All that remains to show is that the primal and dual problems have the same optimal value, and the optimal value in the dual problem is actually attained (i.e., the maximum really is a maximum instead of a supremum). To this end, we simply note that the primal problem is strictly feasible, since we can choose $Y = Z = cI$ for some sufficiently large value of $c$ to make
    \[
        \begin{bmatrix}
            Y & A \\ A^* & Z
        \end{bmatrix} = \begin{bmatrix}
            cI & A \\ A^* & cI
        \end{bmatrix} \succ O.
    \]
    It then follows from Slater's conditions for strong duality that the optimal value of the dual problem is attained, and it equals the optimal value of the primal problem (i.e., $\|A\|_{\textup{S}}$).
\end{proof}

An alternative proof of the above theorem was given by the second author in \cite{MathO14}. That proof used the fact that the Schur norm is a special case of something called a \emph{completely bounded} norm, and all completely bounded norms can be computed via semidefinite programming \cite{Wat09,Wat13}. Code that implements these semidefinite programs for computing Schur norms, via the CVX package for MATLAB \cite{cvx}, or the Convex.jl package for Julia \cite{convexjl}, is available for download from \cite{SuppCode}.

\subsection{Computation of $r_n$}\label{sec:compute_rn}

In Section~\ref{sec:spectra} we computed $r_n$ for a few small values of $n$: $r_2 = \sqrt{2}$, $r_3 = 5/3$, and $r_4 = 2$, and we saw numerical results that suggested that $r_5 = 11/5$ and $r_6 \approx 2.3874$. We now use Theorem~\ref{thm:schur_norm_sdp} to make these numerical results rigorous, and extend them slightly.

\begin{theorem}\label{thm:r5}
    $\displaystyle r_5 = \frac{11}{5}$, $\displaystyle r_6 = \frac{1}{3}(4 + \sqrt{10})$, and $\displaystyle r_7 = \frac{1}{7}(1 + 12\sqrt{2})$.
\end{theorem}

\begin{proof}
    As mentioned in Section~\ref{sec:spectra}, there are only $30$ equivalence classes in $\mathbb{Q}_5(\mathbb{R})$ (i.e., every member of $\mathbb{Q}_5(\mathbb{R})$ can be transformed into one of $30$ fixed matrices via the operations of transposition, multiplication by diagonal unitary matrices, and multiplication by permutation matrices). Furthermore, our code \cite{SuppCode} that implements the semidefinite program described by Theorem~\ref{thm:schur_norm_sdp} shows that $29$ of those $30$ matrices have Schur norm no larger than $2.19$. While this computation was performed numerically, it is accurate to at least $8$ decimal places, so these $29$ matrices are certain to have Schur norm strictly less than $11/5 = 2.2$.
    
    The one matrix (up to the aforementioned equivalences) $A \in \mathbb{Q}_5(\mathbb{R})$ with Schur norm larger than $2.19$ is
    \begin{align*}
        A & = \begin{bmatrix}
            1 & -1 & -1 & -1 & -1 \\
            -1 & 1 & -1 & -1 & -1 \\
            -1 & -1 & 1 & -1 & -1 \\
            -1 & -1 & -1 & 1 & -1 \\
            -1 & -1 & -1 & -1 & 1
        \end{bmatrix}.
    \end{align*}
     Since this matrix $A$ is circulant, $\|A\|_{\textup{S}}$ can be verified to equal exactly $11/5$ via the formula~\eqref{E6}.
     
    A similar computation shows that, of the $4386$ equivalence classes in $\mathbb{Q}_7(\mathbb{R})$, there is only one in which the matrices have Schur norm larger than $2.56 < (1 + 12\sqrt{2})/7 \approx 2.5672$. One matrix in that equivalence class is
    \begin{align}\label{eq:r7_circ}
        B & = \begin{bmatrix}
            1 & 1 & -1 & 1 & -1 & -1 & -1 \\
            -1 & 1 & 1 & -1 & 1 & -1 & -1 \\
            -1 & -1 & 1 & 1 & -1 & 1 & -1 \\
            -1 & -1 & -1 & 1 & 1 & -1 & 1 \\
            1 & -1 & -1 & -1 & 1 & 1 & -1 \\
            -1 & 1 & -1 & -1 & -1 & 1 & 1 \\
            1 & -1 & 1 & -1 & -1 & -1 & 1
        \end{bmatrix}.
    \end{align}
    Since this matrix $B$ is circulant, $\|B\|_{\textup{S}}$ can be verified to equal exactly $(1 + 12\sqrt{2})/7$, again via the formula~\eqref{E6}.
     
    Finally, to demonstrate the value of $r_6$, we showed that matrices in $241$ of the $242$ equivalence classes of $\mathbb{Q}_6(\mathbb{R})$ have Schur norm no larger than $2.38 < \frac{1}{7}(1 + 12\sqrt{2}) \approx 2.3874$. The one matrix (up to the aforementioned equivalences) $C \in \mathbb{Q}_6(\mathbb{R})$ with Schur norm larger than $2.38$ is
    \begin{align*}
        C & = \begin{bmatrix}
            1 &  1 &  1 &  1 &  1 &  1 \\
            1 &  1 &  1 & -1 & -1 & -1 \\
            1 &  1 & -1 &  1 & -1 & -1 \\
            1 & -1 &  1 &  1 & -1 & -1 \\
            1 & -1 & -1 & -1 &  1 & -1 \\
            1 & -1 & -1 & -1 & -1 &  1
        \end{bmatrix}.
    \end{align*}
    
    To see that this matrix $C$ has $\|C\|_{\textup{S}} = (4 + \sqrt{10})/3$, and thus complete the proof, we simply find feasible points of the dual pair of semidefinite programs from Theorem~\ref{thm:schur_norm_sdp} that produce this value in the objective function. It is straightforward to check that the following values of $c$, $Y$, and $Z$ work in the primal problem:
    \[
        c = \frac{1}{3}(4 + \sqrt{10}), \quad Y = Z = \begin{bmatrix}
            c & 0 & 0 & 0 & 2-c & 2-c \\
            0 & c & c-2 & c-2 & 0 & 0 \\
            0 & c-2 & c & c-2 & 0 & 0 \\
            0 & c-2 & c-2 & c & 0 & 0 \\
            2-c & 0 & 0 & 0 & c & c-2 \\
            2-c & 0 & 0 & 0 & c-2 & c
        \end{bmatrix},
    \]
    and the following values of $X$, $\mathbf{v}$, and $\mathbf{w}$ work in the dual problem:
    \[
        X = \frac{1}{18}\begin{bmatrix}
            2 & 0 & 0 & 0 & 1 & 1 \\
            0 & 1 & 1 & -2 & 0 & 0 \\
            0 & 1 & -2 & 1 & 0 & 0 \\
            0 & -2 & 1 & 1 & 0 & 0 \\
            1 & 0 & 0 & 0 & 2 & -1 \\
            1 & 0 & 0 & 0 & -1 & 2
        \end{bmatrix} + \frac{\sqrt{10}}{180}\begin{bmatrix}
            -1 & 3 & 3 & 3 & 1 & 1 \\
            3 & 1 & 1 & 1 & -3 & -3 \\
            3 & 1 & 1 & 1 & -3 & -3 \\
            3 & 1 & 1 & 1 & -3 & -3 \\
            1 & -3 & -3 & -3 & -1 & -1 \\
            1 & -3 & -3 & -3 & -1 & -1
        \end{bmatrix},
    \]
    and $\mathbf{v} = \mathbf{w} = (1,1,1,1,1,1)/6$.
\end{proof}

Computation of $r_n$ in general seems to be a hard problem (which perhaps is not surprising, since a quick method of computing $r_n$ would let us quickly determine whether or not there is an $n \times n$ Hadamard matrix). The method that we used requires us to compute the Schur norm of every matrix in $\mathbb{Q}_n(\mathbb{R})$. While each Schur norm can be computed in polynomial time, we need to compute the Schur norm of exponentially many matrices in $\mathbb{Q}_n(\mathbb{R})$ to determine which one has the largest Schur norm.

In particular, there are $2^{n^2}$ matrices in $\mathbb{Q}_n(\mathbb{R})$. While we do not have to compute the Schur norm of all of them (since the Schur norm does not change upon taking the transpose, permuting rows or columns, or multiplying a row or column by $-1$), there are still at least
\[
    \frac{2^{n^2}}{2 \cdot (n!)^2 \cdot 2^{2n-1}} = \frac{2^{n(n - 2)}}{(n!)^2}
\]
equivalence classes inside $\mathbb{Q}_n(\mathbb{R})$ to check (so, for example, by $n = 15$ there are at least $2.93 \times 10^{34}$ equivalence classes, and we need to compute Schur norms of at least one matrix in each of them).

However, there may be a much more clever way of computing $r_n$, and we do not expect a precise statement concerning the theoretical difficulty of computing it (e.g., we do not expect a proof that computation of $r_n$ is NP-hard). After all, if the Hadamard conjecture is true then it implies $r_{4n} = 2\sqrt{n}$ for all $n$, making $r_{4n}$ trivial to compute.

\subsection{Schur Norms of Circulant Matrices}\label{sec:circulant}

It has been conjectured that there are no circulant Hadamard matrices except in the $n = 1$ and $n = 4$ cases \cite[page~134]{Rys63}, so it seems natural to explore the variant of $r_n$ where, instead of maximizing $\|A\|_{\textup{S}}$ over all members of $\mathbb{Q}_n(\mathbb{R})$, we maximize over all \emph{circulant} members of $\mathbb{Q}_n(\mathbb{R})$. We call the resulting quantity $rC_n$, and our earlier results show that $r_n = rC_n$ when $n \in \{1,3,4,5,7\}$, but $r_n > rC_n$ when $n \in \{2,6,8\}$.

By cataloguing the largest Schur norms of matrices in $\mathbb{Q}_n(\mathbb{R})$ that we have been able to find, we have also shown that $r_n > rC_n$ when $9 \leq n \leq 12$ and when $14 \leq n \leq 24$. The results of these computations, which include the values of $rC_n$ and the best bounds that we have on $r_n$ for $1 \leq n \leq 24$, are summarized in Table~\ref{tab:small_dim_numerics}. The matrices with Schur norms equal to the values are available for download from \cite{SuppCode}.

\begin{table}[!htb]
\begin{center}
    \begin{tabular}{ c | r r r }
    \toprule
    $n$ & \multicolumn{1}{c}{$rC_n$} & \multicolumn{1}{c}{$r_n$ lower bound} & \multicolumn{1}{c}{$r_n$ upper bound} \\ \midrule
    $1$ & \cellcolor{lightgray!30!white} $1 = 1.0000$ & \cellcolor{lightgray!30!white} $1 = 1.0000$ & \cellcolor{lightgray!30!white} $1 = 1.0000$ \\
    $2$ & $1 = 1.0000$ & \cellcolor{lightgray!30!white} $\sqrt{2} \approx 1.4142$ & \cellcolor{lightgray!30!white} $\sqrt{2} \approx 1.4142$ \\
    $3$ & \cellcolor{lightgray!30!white} $5/3 \approx 1.6667$ & \cellcolor{lightgray!30!white} $5/3 \approx 1.6667$ & \cellcolor{lightgray!30!white} $5/3 \approx 1.6667$ \\
    $4$ & \cellcolor{lightgray!30!white} $2 = 2.0000$ & \cellcolor{lightgray!30!white} $2 = 2.0000$ & \cellcolor{lightgray!30!white} $2 = 2.0000$ \\ \midrule
    $5$ & \cellcolor{lightgray!30!white} $11/5 = 2.2000$ & \cellcolor{lightgray!30!white} $11/5 = 2.2000$ & \cellcolor{lightgray!30!white} $11/5 = 2.2000$ \\
    $6$ & $7/3 \approx 2.3333$ & \cellcolor{lightgray!30!white} $(4+\sqrt{10})/3 \approx 2.3874$ & \cellcolor{lightgray!30!white} $(4+\sqrt{10})/3 \approx 2.3874$ \\
    $7$ & \cellcolor{lightgray!30!white} $(1 + 12\sqrt{2})/7 \approx 2.5672$ & \cellcolor{lightgray!30!white} $(1 + 12\sqrt{2})/7 \approx 2.5672$ & \cellcolor{lightgray!30!white} $(1 + 12\sqrt{2})/7 \approx 2.5672$ \\
    $8$ & $1 + \sqrt{3} \approx 2.7321$ & \cellcolor{lightgray!30!white} $2\sqrt{2} \approx 2.8284$ & \cellcolor{lightgray!30!white} $2\sqrt{2} \approx 2.8284$ \\ \midrule
    $9$ & $\approx 2.8539$ & $\approx 2.9477$ & $3 = 3.0000$ \\
    $10$ & $\approx 2.9714$ & $11\sqrt{2}/5 \approx 3.1113$ & $\sqrt{10} \approx 3.1623$ \\
    $11$ & $(1 + 20\sqrt{3})/11 \approx 3.2401$ & $\approx 3.2454$ & $\sqrt{11} \approx 3.3166$ \\
    $12$ & $2+\sqrt{2} \approx 3.4142$ & \cellcolor{lightgray!30!white} $2\sqrt{3} \approx 3.4641$ & \cellcolor{lightgray!30!white} $2\sqrt{3} \approx 3.4641$ \\ \midrule
    $13$ & \cellcolor{lightgray!30!white} $(5+24\sqrt{3})/13 \approx 3.5822$ & \cellcolor{lightgray!30!white} $(5+24\sqrt{3})/13 \approx 3.5822$ & $\sqrt{13} \approx 3.6056$ \\
    $14$ & $(17 + 6\sqrt{2})/7 \approx 3.6408$ & $\approx 3.6977$ & $\sqrt{14} \approx 3.7417$ \\
    $15$ & $\approx 3.8068$ & $\approx 3.8102$ & $\sqrt{15} \approx 3.8730$ \\
    $16$ & $\approx 3.8882$ & \cellcolor{lightgray!30!white} $4 = 4.0000$ & \cellcolor{lightgray!30!white} $4 = 4.0000$ \\ \midrule
    $17$ & $\approx 4.0205$ & $\approx 4.0848$ & $\sqrt{17} \approx 4.1231$ \\
    $18$ & $\approx 4.1265$ & $(32 + \sqrt{34})/9 \approx 4.2034$ & $3\sqrt{2} \approx 4.2426$ \\
    $19$ & $\approx 4.3050$ & $\approx 4.3071$ & $\sqrt{19} \approx 4.3589$ \\
    $20$ & $22/5 = 4.4000$ & \cellcolor{lightgray!30!white} $2\sqrt{5} \approx 4.4721$ & \cellcolor{lightgray!30!white} $2\sqrt{5} \approx 4.4721$ \\ \midrule
    $21$ & $(21 + 4\sqrt{7})/7 \approx 4.5119$ & $\approx 4.5535$ & $\sqrt{21} \approx 4.5826$ \\
    $22$ & $\approx 4.5892$ & $\approx 4.6506$ & $\sqrt{22} \approx 4.6904$ \\
    $23$ & $(1+44\sqrt{6})/23 \approx 4.7295$ & $\approx 4.7426$ & $\sqrt{23} \approx 4.7958$ \\
    $24$ & $\approx 4.8640$ & \cellcolor{lightgray!30!white} $2\sqrt{6} \approx 4.8990$ & \cellcolor{lightgray!30!white} $2\sqrt{6} \approx 4.8990$ \\\bottomrule
    \end{tabular}
    \caption{A summary of the values of $rC_n$, as well as the best bounds on $r_n$ that we have, for $1 \leq n \leq 24$. All numerical values given in the table have been rounded to, and are accurate to, four decimal places. The left column is a lower bound on the middle column, which is a lower bound on the right column. Rows in which two or more columns have the same value are highlighted in gray.}\label{tab:small_dim_numerics}
\end{center}
\end{table}

We note that the values of $rC_n$ given in Table~\ref{tab:small_dim_numerics} have all be rounded to $4$ decimal places, but exact forms for all of them can be found via Equation~\eqref{E6}. We do not provide most of these exact forms, since they are often quite long and ugly. For example, the exact value of $rC_{10}$ is
\[
    rC_{10} = \frac{1}{5}\left(3 + 2\sqrt{5} + 2\sqrt{7 + 2\sqrt{11}}\right) \approx 2.9714,
\]
which is attained at the circulant matrix with top row equal to $(1,1,-1,1,-1,-1,-1,-1,-1,-1)$.

\section{Better Almost Hadamard Matrices}\label{sec:almost_hadamard}

In an approach to the Hadamard conjecture that is somewhat dual to ours, it was shown in \cite{BCS10} that the entrywise $1$-norm (i.e., the sum of absolute values of entries) of an $n \times n$ real orthogonal matrix is bounded above by $n\sqrt{n}$, and equality holds exactly for multiples of Hadamard matrices. Based on this idea, in \cite{BNZ12} the authors asked what the largest entrywise $1$-norm of a real orthogonal matrix is, as a function of its size $n$. They furthermore called an orthogonal matrix that locally maximizes this entrywise $1$-norm an \emph{almost Hadamard matrix}.

The authors found that for $n = 2$, $3$, and $4$, the optimal $1$-norms are $2\sqrt{2}$, $5$, and $8$, respectively, and they made some conjectures based on numerics for $5 \leq n \leq 13$ (see \cite[Table~1]{BNZ12}). In this section, we extend their results by proving the optimal $1$-norms when $5 \leq n \leq 8$, which agree with their conjectures when $n \in \{5,7,8\}$, but disprove their conjecture when $n = 6$. We also disprove their conjectured values when $n \in \{9,11\}$, and we provide numerical results up to $n = 24$.

To give a bit of an idea for how our results in this area work, consider the $6 \times 6$ matrix $X$ from the proof of Theorem~\ref{thm:r5}. It was not important for us back then, but $6X$ is actually an orthogonal matrix. Straightforward computation shows that its entrywise $1$-norm is $8 + 2\sqrt{10} \approx 14.3246$, beating the conjectured largest $1$-norm of $10\sqrt{2} \approx 14.1421$.

The above observation might lead us to believe that we can use the ``dual'' semidefinite program from Theorem~\ref{thm:schur_norm_sdp} to construct almost Hadamard matrices in general. This is almost correct, but the semidefinite program actually needs some slight tweaks:

\begin{theorem}\label{thm:almost_hadamard_sdp}
    Let $A \in \mathbb{Q}_n(\mathbb{R})$, and let $\nu_A$ be the optimal value of the following semidefinite program:
    \begin{align*}
    		\begin{matrix}
    		\begin{tabular}{r l}
    		\textup{maximize:} & $\tr(AX)$ \\
    		\textup{subject to:} & $\begin{bmatrix}
                I & X \\ X^T & I
            \end{bmatrix} \succeq O$.	
    		\end{tabular}
    		\end{matrix}
    \end{align*}
    Then the largest entrywise $1$-norm of an $n \times n$ real orthogonal matrix is equal to
    \begin{align}\label{eq:max_ent_1norm}
        \max\{\nu_{A} : A \in \mathbb{Q}_n(\mathbb{R})\}.
    \end{align}
\end{theorem}

\begin{proof}
    The constraint in the semidefinite program simply forces $\|X\| \leq 1$, so by convexity the maximum is attained when $X$ is an extreme point of the unit ball in the operator norm (i.e., when $X$ is an orthogonal matrix). It follows that the quantity~\eqref{eq:max_ent_1norm} is equal to
    \[
        \max\{ \tr(AX) : A \in \mathbb{Q}_n(\mathbb{R}), X \ \text{is orthogonal}\}.
    \]
    Since $\tr(AX) = \sum_{i,j=1}^n x_{i,j}a_{j,i}$ and $A \in \mathbb{Q}_n(\mathbb{R})$ so $|a_{j,i}| = 1$ for all $i,j$, it is clear that $\tr(AX) \leq \sum_{i,j=1}^n |x_{i,j}|$. Conversely, there exists some $A \in \mathbb{Q}_n(\mathbb{R})$ (matching the sign pattern of $X$) so that $\tr(AX) = \sum_{i,j=1}^n |x_{i,j}|$, which completes the proof.
\end{proof}

An immediately corollary of the above theorem is that the largest entrywise $1$-norm of an $n \times n$ orthogonal matrix is bounded above by $nr_n$ (after all, the SDP described by it is more restrictive than $n$ times the ``dual'' SDP described by Theorem~\ref{thm:schur_norm_sdp}).

The importance of Theorem~\ref{thm:almost_hadamard_sdp} is that it reduces the problem of finding an orthogonal matrix with maximum entrywise $1$-norm to a finite computation: each value of $\nu_A$ can be approximated to any desired accuracy in polynomial time, and there are only finitely many values of $\nu_A$ that need to be computed. Furthermore, since all of the concepts we are discussing (e.g., Schur norms, the entrywise $1$-norm, and the set of real orthogonal matrices) are unchanged under the standard equivalence operations (i.e., permuting rows and/or columns, multiplying rows and/or columns by $-1$, and transposition), it suffices to compute $\nu_A$ for just a single representative of each equivalence class, rather than for all $A \in \mathbb{Q}_n(\mathbb{R})$.

We have performed this computation for $n \leq 7$, so we now know the exact largest entrywise $1$-norm of a real orthogonal matrix in these cases. We have also use this method to search numerically for orthogonal matrices with large entrywise $1$-norm for $n \leq 24$. Our results are summarized in Table~\ref{tab:almost_hadamard}, and the orthogonal matrices that attain these lower bounds can be downloaded from \cite{SuppCode}.

\begin{table}[!htb]
\begin{center}
    \begin{tabular}{ c | r r c }
    \toprule
    $n$ & \multicolumn{1}{c}{entrywise $1$-norm lower bound} & \multicolumn{1}{c}{upper bound} & \multicolumn{1}{c}{notes} \\ \midrule
    \rowcolor{lightgray!30!white} $1$ & $1 = 1.0000$ & $1 = 1.0000$ & Hadamard \\
    \rowcolor{lightgray!30!white} $2$ & $2\sqrt{2} \approx 2.8284$ & $2\sqrt{2} \approx 2.8284$ & Hadamard \\
    \rowcolor{lightgray!30!white} $3$ & $5 = 5.0000$ & $5 = 5.0000$ & \\
    \rowcolor{lightgray!30!white} $4$ & $8 = 8.0000$ & $8 = 8.0000$ & Hadamard \\\midrule
    \rowcolor{lightgray!30!white} $5$ & $11 = 11.0000$ & $11 = 11.0000$ & \\
    \rowcolor{lightgray!30!white} $6$ & $8 + 2\sqrt{10} \approx 14.3246$ & $8 + 2\sqrt{10} \approx 14.3246$ & better than previously known\\
    \rowcolor{lightgray!30!white} $7$ & $1 + 12\sqrt{2} \approx 17.9706$ & $1 + 12\sqrt{2} \approx 17.9706$ & \\
    \rowcolor{lightgray!30!white} $8$ & $16\sqrt{2} \approx 22.6274$ & $16\sqrt{2} \approx 22.6274$ & Hadamard \\\midrule
    $9$ & $\approx 26.5204$ & $27 = 27.0000$ & better than previously known\\
    $10$ & $22\sqrt{2} \approx 31.1127$ & $10\sqrt{10} \approx 31.6228$ & \\
    $11$ & $\approx 35.6991$ & $11\sqrt{11} \approx 36.4828$ & better than previously known\\
    \rowcolor{lightgray!30!white} $12$ & $24\sqrt{3} \approx 41.5692$ & $24\sqrt{3} \approx 41.5692$ & Hadamard\\\midrule
    $13$ & $5 + 24\sqrt{3} \approx 46.5692$ & $13\sqrt{13} \approx 46.8722$ & \\
    $14$ & $\approx 51.7673$ & $14\sqrt{14} \approx 52.3832$ & \\
    $15$ & $\approx 57.1522$ & $15\sqrt{15} \approx 58.0948$ & \\
    \rowcolor{lightgray!30!white} $16$ & $64 = 64.0000$ & $64 = 64.0000$ & Hadamard\\\midrule
    $17$ & $60 + \sqrt{89} \approx 69.4340$ & $17\sqrt{17} \approx 70.0928$ & \\
    $18$ & $64 + 2\sqrt{34} \approx 75.6619$ & $54\sqrt{2} \approx 76.3675$ & \\
    $19$ & $\approx 81.8353$ & $19\sqrt{19} \approx 82.8191$ & \\
    \rowcolor{lightgray!30!white} $20$ & $4\sqrt{5} \approx 89.4427$ & $4\sqrt{5} \approx 89.4427$ & Hadamard\\\midrule
    $21$ & $\approx 95.6206$ & $21\sqrt{21} \approx 96.2341$ & \\
    $22$ & $\approx 102.3127$ & $22\sqrt{22} \approx 103.1891$ & \\
    $23$ & $\approx 109.0784$ & $23\sqrt{23} \approx 110.3041$ & \\
    \rowcolor{lightgray!30!white} $24$ & $48\sqrt{6} \approx 117.5755$ & $48\sqrt{6} \approx 117.5755$ & Hadamard\\\bottomrule
    \end{tabular}
    \caption{A summary of the best lower and upper bounds on the maximum entrywise $1$-norm of an $n \times n$ real orthogonal matrix that we have been able to compute, for $1 \leq n \leq 24$. All numerical values have been rounded to, and are accurate to, four decimal places. Rows highlighted in gray indicate that the lower and upper bounds are equal, so the exact maximum value is known.}\label{tab:almost_hadamard}
\end{center}
\end{table}

\section{Conclusions and Open Questions}\label{sec:conclusions}

In this work, we initiated the study of Schur norms of matrices whose entries are $\pm 1$, with the goal of making progress on finding the largest ratio by which a matrix's operator norm can increase when decreasing the modulus of its entries (i.e., the value of the quantity~\eqref{E1}). We showed that this quantity equals $\sqrt{n}$ if and only if there is a Hadamard matrix of order $n$, so computing it in general is probably extremely difficult. However, we presented an algorithm that allows for its computation for small $n$, and we computed it exactly when $n \leq 8$.

We then used our techniques to improve upon known results about almost Hadamard matrices. We have left numerous questions open, some of which we summarize here:

\begin{itemize}
    \item We expect that $r_n$ is attained by members of $\Q_n(\R)$ that are ``close to orthogonal'' in some sense, but it is not clear exactly how to make that precise. For example, it is \emph{not} true that $r_n$ is attained by a member of $\Q_n(\R)$ with maximum determinant: we showed in Theorem~\ref{thm:r5} that all maximizers of $r_7$ are equivalent to the circulant matrix~\eqref{eq:r7_circ} and thus have determinant with absolute value $512$, whereas the maximum determinant of a matrix in $\Q_7(\R)$ is $576$. Is it at least true that $r_n$ is always attained by an invertible member of $\Q_n(\R)$?
    
    \item We showed in Lemma~\ref{lem:rn_nondec} that the sequence $\{r_n\}$ is non-decreasing, but we expect that it is actually strictly increasing. Numerics from Table~\ref{tab:small_dim_numerics} suggest that $r_n > \sqrt{n-1}$ which, if true, would imply strict monotonicity of $\{r_n\}$.
    
    \item The argument used to find the decomposition in Equation~\eqref{eq:M3_norm} shows that if $M$ has only $k$ different rows (or columns), then $\|M\|_{\textup{S}} \leq \sqrt{k}$. If the conjecture that $r_n > \sqrt{n-1}$ is true, then this would imply that $r_n$ is always attained by a matrix with distinct rows and distinct columns.
    
    \item We have been unable to find a matrix in $\mathbb{Q}_{13}(\mathbb{R})$ that exhibits a gap between $rC_{13}$ and $r_{13}$; a rough numerical search suggests that $r_{13}$ might be attained by a circulant matrix. At first this seems somewhat strange, but something similar happens when considering the maximum determinant of a member of $\mathbb{Q}_n(\mathbb{R})$: it is attained by a circulant matrix when $n = 13$ (see \cite{oeisA003433,oeisA215723}, for example). Based on these numerics, we conjecture that $r_n = rC_n$ if and only if $n \in \{1,3,4,5,7,13\}$.
\end{itemize}

\noindent \textbf{Acknowledgements.} Several colleagues have commented on an earlier version of this note. We thank, in particular, Rajendra Bhatia, Erik Christensen, Ken Davidson, Milan Hladnik, and Gilles Pisier. N.J.\ was supported by NSERC Discovery Grant RGPIN-2022-04098.

\bibliographystyle{alpha}
\bibliography{bib}

\newcommand{\etalchar}[1]{$^{#1}$}
\begin{thebibliography}{UMZ{\etalchar{+}}14}

\bibitem[BCD89]{BCD89}
R.~Bhatia, M.-D. Choi, and C.~Davis.
\newblock Comparing a matrix to its off-diagonal part.
\newblock In H.~Dym, S.~Goldberg, M.~A. Kaashoek, and P.~Lancaster, editors,
  {\em The Gohberg Anniversary Collection: Volume I: The Calgary Conference and
  Matrix Theory Papers and Volume II: Topics in Analysis and Operator Theory},
  pages 151--164. Birkh{\"a}user Basel, 1989.

\bibitem[BCS10]{BCS10}
T.~Banica, B.~Collins, and J.-M. Schlenker.
\newblock On orthogonal matrices maximizing the 1-norm.
\newblock {\em Indiana University Mathematics Journal}, 59:839--856, 2010.

\bibitem[Bha07]{Bha97}
R.~Bhatia.
\newblock {\em Positive Definite Matrices}.
\newblock Princeton University Press, 2007.

\bibitem[BNZ12]{BNZ12}
T.~Banica, I.~Nechita, and K.~\.{Z}yczkowski.
\newblock Almost {H}adamard matrices: General theory and examples.
\newblock {\em Open Systems \& Information Dynamics}, 19:1250024, 2012.

\bibitem[DD07]{DD07}
K.~R. Davidson and A.~P. Donsig.
\newblock Norms of {S}chur multipliers.
\newblock {\em Illinois Journal of Mathematics}, 51(3):743--766, 2007.

\bibitem[GB14]{cvx}
Michael Grant and Stephen Boyd.
\newblock {CVX}: Matlab software for disciplined convex programming, version
  2.1.
\newblock \url{http://cvxr.com/cvx}, March 2014.

\bibitem[HJS22]{SuppCode}
J.~Holbrook, N.~Johnston, and J.-P. Schoch.
\newblock {MATLAB} code, {J}ulia code, and other supplementary material.
\newblock
  \url{http://www.njohnston.ca/publications/schur-norms-hadamard-matrices/}.
  Also available in the ``source'' files for the arXiv version of this paper,
  2022.

\bibitem[Hla99]{Hla99}
M.~Hladnik.
\newblock Schur norms of bicirculant matrices.
\newblock {\em Linear Algebra and its Applications}, 286:261--272, 1999.

\bibitem[Joh14]{MathO14}
Nathaniel Johnston.
\newblock How hard ({P}, {NP}, {NP}-hard) is it to compute {S}chur norms of
  matrices (as multipliers)?
\newblock MathOverflow, 2014.
\newblock URL: https://mathoverflow.net/q/167642.

\bibitem[Joh21]{JohALA}
N.~Johnston.
\newblock {\em Advanced Linear and Matrix Algebra}.
\newblock Springer International Publishing, 2021.

\bibitem[Lov06]{Lov06}
L.~Lov\'{a}sz.
\newblock {\em Semidefinite {P}rograms and {C}ombinatorial {O}ptimization},
  volume~11, pages 137--194.
\newblock Springer, New York, 2006.

\bibitem[Mat93a]{Mat93}
R.~Mathias.
\newblock The {H}adamard operator norm of a circulant and applications.
\newblock {\em SIAM Journal on Matrix Analysis and Applications},
  14:1152--1167, 1993.

\bibitem[Mat93b]{Mat93b}
R.~Mathias.
\newblock Matrix completions, norms, and {H}adamard products.
\newblock {\em Proceedings of the American Mathematical Society},
  117(4):905--918, 1993.

\bibitem[Pal33]{Pal33}
R.~E.~A.~C. Paley.
\newblock On orthogonal matrices.
\newblock {\em Journal of Mathematics and Physics}, 12:311--320, 1933.

\bibitem[PPS89]{PPS89}
V.~Paulsen, S.~Power, and R.~Smith.
\newblock Schur products and matrix completions.
\newblock {\em Journal of Functional Analysis}, 85:151--178, 1989.

\bibitem[Rys63]{Rys63}
H.~J. Ryser.
\newblock {\em Combinatorial Mathematics}.
\newblock The Carus Mathematical Monographs. Mathematical Association of
  America, 1963.

\bibitem[Sca98]{Sca98}
U.~Scarpis.
\newblock Sui determinanti di valore massimo.
\newblock {\em Rendiconti della R. Istituto Lombardo di Scienze e Lettere},
  31:1441--1446, 1898.

\bibitem[Sch11]{Sch11}
J.~Schur.
\newblock Bemerkungen zur {T}heorie der beschränkten {B}ilinearformen mit
  unendlich vielen {V}eränderlichen.
\newblock {\em Journal für die reine und angewandte Mathematik}, 140:1--28,
  1911.

\bibitem[Slo10]{oeisA003433}
Neil J.~A. Sloane.
\newblock Sequence {A003433} in \emph{{T}he {O}n-{L}ine {E}ncyclopedia of
  {I}nteger {S}equences}.
\newblock \url{https://oeis.org/A003433}, 2010.
\newblock Hadamard maximal determinant problem: largest determinant of
  $(+1,-1)$-matrix of order $n$. [Online; accessed May 21, 2022].

\bibitem[Slo12]{oeisA215723}
Neil J.~A. Sloane.
\newblock Sequence {A215723} in \emph{{T}he {O}n-{L}ine {E}ncyclopedia of
  {I}nteger {S}equences}.
\newblock \url{https://oeis.org/A215723}, 2012.
\newblock Maximum determinant of an $n \times n$ circulant $(1,-1)$-matrix.
  [Online; accessed May 21, 2022].

\bibitem[Slo22]{oeisA353052}
Neil J.~A. Sloane.
\newblock Sequence {A353052} in \emph{{T}he {O}n-{L}ine {E}ncyclopedia of
  {I}nteger {S}equences}.
\newblock \url{https://oeis.org/A353052}, 2022.
\newblock Number of inequivalent $\{-1,1\}$ matrices of order n, up to
  permutation of rows and/or columns, multiplication of rows and/or columns by
  $-1$, and transposition. [Online; accessed May 20, 2022].

\bibitem[Syl67]{Syl67}
J.~J. Sylvester.
\newblock Thoughts on inverse orthogonal matrices, simultaneous sign
  successions, and tessellated pavements in two or more colours, with
  applications to newton's rule, ornamental tile-work, and the theory of
  numbers.
\newblock {\em Philosophical Magazine}, 34:461--475, 1867.

\bibitem[UMZ{\etalchar{+}}14]{convexjl}
Madeleine Udell, Karanveer Mohan, David Zeng, Jenny Hong, Steven Diamond, and
  Stephen Boyd.
\newblock Convex optimization in {J}ulia.
\newblock {\em SC14 Workshop on High Performance Technical Computing in Dynamic
  Languages}, 2014.

\bibitem[Wat09]{Wat09}
J.~Watrous.
\newblock Semidefinite programs for completely bounded norms.
\newblock {\em Theory of Computing}, 5:217--238, 2009.

\bibitem[Wat13]{Wat13}
J.~Watrous.
\newblock Simpler semidefinite programs for completely bounded norms.
\newblock {\em Chicago Journal of Theoretical Computer Science}, 8:1--19, 2013.

\bibitem[Wat18]{Wat18}
J.~Watrous.
\newblock {\em The Theory of Quantum Information}.
\newblock Cambridge University Press, 2018.

\end{thebibliography}

\end{document}